\documentclass{article}

\usepackage{microtype}
\usepackage{verbatim}
\usepackage{graphicx}
\usepackage{rotating}
\usepackage{tikz-cd}
\usepackage{float}

\usepackage{url}
\usepackage{hyperref}

\usepackage{pslatex}

\usepackage{datetime}
\newdateformat{versiondate}{%
\THEMONTH\THEDAY}

\usepackage{xpatch} 

\usepackage{amsthm}
\usepackage[fleqn]{amsmath} % flush equations left
\usepackage{amsfonts}
\usepackage{amssymb}
\usepackage{stmaryrd} % for big square caps/cups

\newtheoremstyle{zoltanstyle}
  {1em} % Space above
  {\topsep} % Space below, usually \topsep
  {} % Body font
  {} % Indent amount
  {\bfseries} % Theorem head font \bfseries
  {.} % Punctuation after theorem head
  {.5em} % Space after theorem head
  {} % Theorem head spec (can be left empty, meaning `normal')

\theoremstyle{zoltanstyle}
\swapnumbers
\xpatchcmd\swappedhead{~}{.~}{}{}
\newtheorem{body}{}
\numberwithin{body}{section}

\newtheorem{corollary}[body]{Corollary}
\newtheorem{definition}[body]{Definition}
\newtheorem{example}[body]{Example}

\newtheorem{lemma}[body]{Lemma}

\newtheorem{proposition}[body]{Proposition}

\newtheorem{theorem}[body]{Theorem}

\expandafter\let\expandafter\oldproof\csname\string\proof\endcsname
\let\oldendproof\endproof

\renewenvironment{proof}[1][\proofname]{%
  \oldproof[\normalfont \bfseries #1.]%
}{\oldendproof}

\newcommand{\vdashkm}{\vdash_{\mathbf{KM}}}
\newcommand{\Hh}{\mathcal{H}}
\newcommand{\SetComp}[2]{\left\{ {#1}\:\middle|\:{#2} \right\}}   % ZFC comprehension
\newcommand{\pita}{\mathrel{\reflectbox{\rotatebox[origin=c]{180}{$\mathbb{A}$}}\!\!}}  % PITts forAll
\newcommand{\pite}{\mathrel{\reflectbox{$\mathbb{E}$}\!\!}}  % PITts Exists

\newcommand{\dera}{\nabla\langle a\rangle}

\title{Two applications of the point-free coderivative}
\author{Zoltan A. Kocsis}
\date{24 September 2026}

\begin{document}

\maketitle

\begin{abstract}
\noindent We present two new applications of Simmons' point-free Cantor-Bendixson coderivative operator in intuitionistic logic. First, we use it to give a simplified proof of the recent result of Xu and Ye that the free Heyting algebra on two generators does not occur as the Heyting algebra of subterminal objects in any elementary topos. Then we use it to prove that complete Heyting algebra semantics is not strongly complete for intuitionistic second-order propositional logic: semantic consequence from an arbitrary set of assumptions does not coincide with ordinary syntactic consequence.
\end{abstract}

\section{Introduction}

In this article, we give two fresh applications of the point-free coderivative operation
$$ \nabla(x) = \bigsqcap_{y \in H} y \sqcup (y \Rightarrow x)$$
on complete Heyting algebras $(H, \sqsubseteq)$, introduced by Simmons~\cite{simmons-cantor} in 1982.

\vspace{0.5em} The first application, presented in Section~\ref{sec:xu-ye-theorem}, is a new, simplified proof of the recent result of Xu and Ye~\cite{xu-ye-failure} that the free Heyting algebra on two generators canont occur as the Heyting algebra of subterminal objects (``truth values'') of any elementary topos. A result of Bellissima embeds the free Heyting algebra on two generators $F_2$ into a complete Heyting algebra $\mathfrak{H}_2$. We call an element of $\mathfrak{H}_2$ \textit{propositionally representable} if it lies in the image of this embedding. We show that the point-free coderivative $\dera$ of a free generator $a \in F_2$ is not propositionally representable, but the internal language of an elementary topos whose algebra of truth values is $F_2$ would have to contain a truth value whose image in $\mathfrak{H}_2$ is $\dera$.

\vspace{0.5em} The second application, presented in Section~\ref{sec:failure-of-completeness}, is a proof that the complete Heyting algebra semantics of intuitionistic second-order propositional logic is not \textit{strongly complete}: semantic consequence from an infinite set of assumptions does not coincide with ordinary syntactic consequence. The semantics interprets the propositional quantifiers $\forall P, \exists P$ as meets and joins over the algebra. This means that the unary connective
$$\Delta(X) := \forall Q. Q \vee (Q \rightarrow X)$$
of second-order propositional logic corresponds to the $\nabla$ operator of Simmons on the algebra. Writing $\vdash_2$ for the syntactic consequence relation of intuitionistic second-order propositional logic, we construct a set of second-order propositional formulas $\mathcal{T}$ involving $\Delta$ so that $\mathcal{T} \vDash \varphi$ but $\mathcal{T} \not\vdash_2 \varphi$, thereby showing that complete Heyting algebra semantics is not strongly complete.

\vspace{0.5em}  Both applications were originally suggested by F{\'e}r{\'e}e and Shillito's recent uniform interpolation theorem~\cite{feree-shillito} for the intuitionistic modal logic of Kuznetsov and Muravitsky~\cite{kuznetsov-km-logic}, whose standard Kripke semantics is itself not strongly complete. At the end of Section~\ref{sec:failure-of-completeness}, we discuss an alternative proof-theoretic version of the second application, which can be obtained directly from the uniform interpolation theorem.

\vspace{0.5em} Complete Heyting algebra semantics has often been claimed to be complete for intuitionistic second-order propositional logic: for example, \textit{Lectures on the Curry-Howard Isomorphism}~\cite{urzyczyn-lectures} states strong completeness while citing Geuvers~\cite{geuvers-semantics}, whose 1994 article states only weak completeness (every formula valid in every complete Heyting algebra model is derivable). In 2009, Zdanowski~\cite{zdanowski-quant} pointed out a serious flaw even in the weak completeness argument: Geuevers assumes that passing to a completion preserves the range of propositional quantification. Our application here concerns strong completeness only. To the best knowledge of the author, the weak completeness question is still unresolved.

\paragraph{Outline.} Section~\ref{sec:notation} fixes the common background and notational conventions, Section~\ref{sec:bellissima} provides an introduction to Bellissima's construction, while Sectionss~\ref{sec:xu-ye-theorem} and \ref{sec:failure-of-completeness} contain the two applications and can be read independently of one another.

\section{Preliminaries}\label{sec:preliminaries}

\begin{body}
In this section we fix the notational conventiuons used throughout the article. After recalling the prerequisites, we introduce complete Heyting algebras and the point-free coderivative, and define complete Heyting algebra semantics for intuitionistic second-order propositional logic.
\end{body}

\subsection{Notation}\label{sec:notation}

\begin{body}
We assume that the reader is familiar with both Heyting-valued and Kripke semantics for intuitionistic propositional logic. For a self-contained presentation of these semantics and their relationships, we refer the reader to \textit{A semantic hierarchy for intuitionistic logic}~\cite{bezhanishvili-hierarchy}. The results of Section~\ref{sec:xu-ye-theorem} require some familiarity with the internal language of topoi as well: as a reference, we recommend Part~III of \textit{Elementary Categories, Elementary Toposes}~\cite{mclarty-topoi}.
\end{body}

\begin{body}
In this article, we use the symbols $\wedge$ / $\vee$ / $\rightarrow$ only in logical syntax (standing for conjunction / disjunction / implication respectively), never as operations defined on other lattices. The relations and operations of Heyting algebras are denoted by variants of the symbols $\sqsubseteq, \sqcap, \sqcup, \Rightarrow$, to avoid clashes with set-theoretic and logical notation.
\end{body}

\begin{body}\label{body:lattice-notation}
Recall that in a Heyting algebra, any one of the binary operations is sufficient to define the partial order relation, and hence the other binary operations:
$$ x \sqsubseteq y \quad\leftrightarrow\quad x \sqcap y = x \quad\leftrightarrow\quad x \sqcup y = y\quad\leftrightarrow\quad x\Rightarrow y=x\Rightarrow x.$$
In light of this, we adopt the following notational convention. We denote Heyting algebras using only their partial order, as $(X, \sqsubseteq_X)$. In any result concerning such a Heyting algebra, we let $\sqcup_X$ denote the least upper bound map definable from $\sqsubseteq_X$. Similarly, $\sqcap_X$ denotes the corresponding greatest lower bound map, $\Rightarrow_X$ the Heyting implication, $\top_X$ the $\sqsubseteq$-maximum element of $X$, and $\bot_X$ to the $\sqsubseteq$-minimum element of $X$. For example, in a result concerning the three Heyting algebras $(X, \sqsubseteq_X)$, $(Y, \sqsubseteq_Y)$ and $(H, \sqsubseteq)$, under our convention
\begin{itemize}
    \item the symbol $\sqcap_X$ would denote the greatest lower bound function $X^2 \rightarrow X$ with respect to the partial order $\sqsubseteq_X$,
    \item the symbol $\bot_Y$ would denote the minimum element of $Y$ with respect to the partial order $\sqsubseteq_Y$,
    \item the symbol $\sqcup$ would denote the least upper bound function $H^2 \rightarrow H$ with respect to the partial order $\sqsubseteq$,
\end{itemize}
and so on. 
\end{body}

\subsection{Complete Heyting algebras}

\begin{definition}\label{def:complete-heyting-algebra}
We call the Heyting algebra $(H, \sqsubseteq)$ a \textit{complete Heyting algebra} if every subset $S \subseteq H$ has a $\sqsubseteq$-least upper bound and $\sqsubseteq$-greatest lower bound. We denote such a least upper bound as $\bigsqcup_{x \in S} x$.
\end{definition}

\begin{example}
The open set lattice $(\Omega T, \subseteq)$ of any topological space $T$, ordered by inclusion, constitutes a complete Heyting algebra.
\end{example}

\subsection{The coderivative operator $\nabla$}

\begin{body}
The \textit{point-free coderivative}, which we denote by $\nabla$ below, was first devised by Simmons~\cite{simmons-cantor} for the study of scattered topological spaces. For an overview of its many uses in topology and logic, see the review article of Litak~\cite{litak-modalities}.
\end{body}

\begin{definition}\label{def:nabla}
Take a complete Heyting algebra $(H, \sqsubseteq)$. We call the map $\nabla : H \rightarrow H$ given by the equality
$$ \nabla(x) = \bigsqcap_{y \in H} y \sqcup (y \Rightarrow x)$$
the \textit{point-free coderivative operator on $(H, \sqsubseteq)$}.
\end{definition}

\begin{body}
The Cantor-Bendixson derivative is one of the oldest operations studied in point-set topology. It associates to each set $S$ its set of accumulation points, ino ther words the set of those $x$ whose every open neighbourhood meets $S \setminus \{x\}$. The coderivative is the corresponding dual operator, which associates to $S$ the set of those $x$ which have an open neighborhood contained within $S \cup \{x\}$. For $T_0$-spaces, the Cantor-Bendixson coderivative coincides with $\nabla$. % nb this is Theorem 22 of Litak's article.
\end{body}

\begin{body}\label{body:alexandroff-construction}
Every preorder $(K, \leq)$ gives rise to an Alexandroff topology on $K$ by defining a set $V \subseteq K$ as open precisely if it is upward-closed ($\forall x. \forall y. x \leq y \rightarrow x \in V \rightarrow y \in V$). We call an upward-closed set an \textit{upset} for short. On such spaces, the point-free coderivative admits a particularly simple characterization, stated as Proposition~\ref{prop:nabla-characterization-upsets} below.
\end{body}

\subsection{Free Heyting algebras}

\begin{body}
Recall that the \textit{free Heyting algebra on two generators}, denoted $F_2$, is the Lindenbaum algebra of intuitionistic propositional logic in two variables. Its elements are equivalence classes of formulas of intuitionistic propositional logic, quotiented by the relation $\varphi_1 \sim \varphi_2$ which holds precisely if $\varphi_1$ and $\varphi_2$ are provably equivalent, i.e.
$$ \vdash \varphi_1 \leftrightarrow \varphi_2$$
holds in intuitionistic logic. The Heyting operations $(\bot, \sqcap, \sqcup, \Rightarrow)$ on $F_2$ are induced by the logical connectives $(\bot, \wedge, \vee, \rightarrow)$. Throughout the article we fix the free generators of $F_2$ as $a,b$, and write $[\varphi] \in F_2$ for the $\sim$-class of the intuitionistic propositional formula $\varphi$ in these two variables.
\end{body}

\begin{body}
A long-standing question in mathematical logic asked whether every Heyting algebra can occur as the algebra of truth values of some elementary topos. Recently, Xu and Ye~\cite{xu-ye-failure} resolved this question in the negative by proving that $F_2$ does not occur as such an algebra of truth values. The subject of Section~\ref{sec:xu-ye-theorem} is a new proof of this result, which replaces the original proof's concrete obstruction term with a more transparent one based on the point-free coderivative operator of Simmons. This simplifies the argument, eliminating a chunk of combinatorial calculations involving mutually recursive sequences, and the corresponding six-component internal encoding and successor operation.
\end{body}

\subsection{Intuitionistic second-order propositional logic}

\begin{body}
Intuitionistic second-order propositional logic \textbf{LJ2} is obtained from intuitionistic propositional logic \textbf{LJ} by adding quantifiers which bind propositional variables. We present \textbf{LJ2} below using a sequent system.
Although the connectives $\wedge,\vee,\bot$ and even the quantifier $\exists$ admit definitions in terms of $\rightarrow$ and $\forall$ in second-order logic (see~\cite{urzyczyn-lectures}),  we choose a presentation which includes the logical connectives $\wedge, \vee, \rightarrow, \bot$ and both second-order quantifiers $\forall, \exists$ as primitives. This ensures that the quantifier-free fragment of our second-order logic corresponds exactly to ordinary intuitionistic propositional logic.  We define the notion of \textit{free and bound occurrences} of variables in formulas, and the \textit{simultaneous substitution} of terms $T_1,\dots,T_n$ for distinct variables $X_1,\dots,X_n$ (denoted $A[T_1/X_1,\dots,T_n/X_n]$) in a capture-avoiding manner in the obvious way.  We introduce the connectives $\top, \neg, \leftrightarrow$ as abbreviations, however: $\top$ stands for $\bot \rightarrow \bot$, $\neg \varphi$ stands for $\varphi \rightarrow \bot$,  and $\varphi \leftrightarrow \psi$ standing for $(\varphi\rightarrow \psi) \wedge (\psi \rightarrow \varphi)$.
\end{body}

\begin{definition}\label{def:lj2-calculus}A \textit{sequent} is an expression $\Gamma \vdash_2 A$ where $A$ is a formula and $\Gamma$ denotes a finite (possibly empty) sequence of formulas, considered up to order. A sequent $\Gamma$ with empty right-hand side denotes $\Gamma \vdash_2 \bot$. We follow the same presentation as Hermant-Lipton~\cite{hermant-lipton} and as the previous works of the author~\cite{kocsis-apal}. \\
\begin{tabular}{lllll}
 \textbf{Identity} & & & \\
 & $\frac{~}{\varphi \vdash_2 \varphi}\text{ax}$ & $\frac{\Gamma \vdash_2 \varphi ~~~~ \Lambda, \varphi \vdash_2 \psi}{\Gamma, \Lambda \vdash_2 \psi}\text{cut}$ \\ \\
 \textbf{Structure} & & & \\
 & $\frac{\Gamma \vdash_2 \psi}{\Gamma, \varphi \vdash_2 \psi}wL$ & $\frac{\Gamma, \varphi, \varphi \vdash_2 \psi}{\Gamma, \varphi \vdash_2 \psi}cL$ & $\frac{\Gamma \vdash_2}{\Gamma \vdash_2 \varphi}wR$ \\ \\
 \textbf{Connectives} & & & \\
 & % Disj
 $\frac{\Gamma, \varphi \vdash_2 \theta ~~~~ \Gamma, \psi \vdash_2 \theta}{\Gamma, \varphi \vee \psi \vdash_2 \theta}\vee L$ &
 $\frac{\Gamma \vdash_2 \varphi}{\Gamma \vdash_2 \varphi \vee \psi}\vee R_1$ &
 $\frac{\Gamma \vdash_2 \psi}{\Gamma \vdash_2 \varphi \vee \psi}\vee R_2$
 \\ \\
 & % Conj
 $\frac{\Gamma \vdash_2 \varphi ~~~~ \Gamma \vdash_2 \psi}{\Gamma \vdash_2 \varphi \wedge \psi}\wedge R$ &
 $\frac{\Gamma, \varphi \vdash_2 \theta}{\Gamma, \varphi \wedge \psi \vdash_2 \theta}\wedge L_1$ &
 $\frac{\Gamma, \psi \vdash_2 \theta}{\Gamma, \varphi \wedge \psi \vdash_2 \theta}\wedge L_2$
 \\ \\
 & % Impl
 $\frac{\Gamma \vdash_2 \varphi ~~~~ \Lambda, \psi \vdash_2 \theta}{\Gamma, \Lambda, \varphi \rightarrow \psi \vdash_2 \theta}\rightarrow L$ &
 $\frac{\Gamma, \varphi \vdash_2 \psi}{\Gamma \vdash_2 \varphi \rightarrow \psi}\rightarrow R$ &
 $\frac{~}{\bot \vdash_2 \varphi}\bot L$
 \\ \\
 \textbf{Quantifiers} & & & \\
 & % Univ
 $\frac{\Gamma, \varphi[T/Y] \vdash_2 \psi}{\Gamma, \forall Y. \varphi \vdash_2 \psi}\forall L$ &
 $\frac{\Gamma \vdash_2 \varphi}{\Gamma \vdash_2 \forall Y. \varphi}\forall R$
 \\ \\
 & % Exis
 $\frac{\Gamma \vdash_2 \varphi[T/Y]}{\Gamma \vdash_2 \exists Y. \varphi}\exists R$ &
 $\frac{\Gamma, \varphi \vdash_2 \psi}{\Gamma, \exists Y. \varphi \vdash_2 \psi}\exists L$
 \\ \\
\end{tabular}
~ \\
with $\Gamma, \Lambda$ standing for arbitrary sequences of formulas, $\varphi,\psi,\theta$ standing for arbitrary formulas, and $Y$ standing for a propositional variable symbol. The rules for the quantifiers satisfy the usual bound variable conditions, for example in the rules $\exists L$ and $\forall R$, the variable $Y$ must not occur free in the context $\Gamma$, and in $\exists L$ it must not occur free in $\psi$ either.
\end{definition}

\begin{definition}
Let $\Gamma$ be a set of second-order propositional formulas, and $\varphi$ one such formula. We say that \textit{$\varphi$ is a syntactic consequence of $\Gamma$}, and write (with some abuse of notation) $\Gamma \vdash_2 \varphi$ if there is a finite sequence of formulas $\Lambda$ so that
\begin{enumerate}
    \item every formula $\psi$ that occurs in $\Lambda$ belongs to the set $\Gamma$, and
    \item there is some derivation tree in the calculus \textbf{LJ2} with conclusion $\Lambda \vdash_2 \varphi$.
\end{enumerate}
\end{definition}

\subsection{Complete Heyting algebra semantics}

\begin{body}
In this section we briefly review the semantics of second-order propositional logic in complete Heyting algebras. Our development closely follows Chapter~12 of Urzyczyn and Sorensen's \textit{Lectures on the Curry-Howard isomorphism}~\cite{urzyczyn-lectures}.
\end{body}

\begin{definition}\label{def:valuation}
A \textit{valuation $v$ on the set $H$} is a map from propositional variable letters to elements of $H$.
\end{definition}

\begin{body}
We will often write $\mathcal{H}$ for a complete Heyting algebra model. When this causes no ambiguity, we assume that $(H, \sqsubseteq)$ denotes the underlying Heyting algebra of the cHA model $\mathcal{H}$. 
\end{body}

\begin{definition}\label{def:cha-interpretation}
Given a complete Heyting algebra $(H,\sqsubseteq)$, we define the interpretation map $\langle-\rangle^H_v$ from \textbf{LJ2}-formulas to elements of $H$, simultaneously on every $H$-valuation $v$, by recursion on the structure of formulas, using the following clauses:
\begin{enumerate}
  \item $\langle X \rangle^H_v := v(X)$ for propositional variable letters $X$,
  \item $\langle \bot \rangle^H_v := \bot_H$,
  \item $\langle \varphi \wedge \psi \rangle^H_v := \langle \varphi \rangle^H_v \sqcap \langle \psi \rangle^H_v$,
  \item $\langle \varphi \vee \psi \rangle^H_v := \langle \varphi \rangle^H_v \sqcup \langle \psi \rangle^H_v$,
  \item $\langle \varphi \rightarrow \psi \rangle^H_v := \langle \varphi \rangle^H_v \Rightarrow \langle \psi \rangle^H_v$,
  \item $\langle \forall Y.\varphi \rangle^H_v := \bigsqcap_{y \in H}\langle \varphi \rangle^H_{v(Y \mapsto y)}$ and
  \item $\langle \exists Y.\varphi \rangle^H_v := \bigsqcup_{y \in H}\langle \varphi \rangle^H_{v(Y \mapsto y)}$,
\end{enumerate}
where $v(Y \mapsto y)$ denotes the valuation $v'$ which satisfies the equations $v'(Y) = y$ and $v'(X) = v(X)$ for all $X \neq Y$. 
\end{definition}

\begin{definition}\label{def:cha-model}
A \textit{complete Heyting algebra model} (or \textit{cHA model} for short) is given by a complete Heyting algebra $(H, \sqsubseteq)$, together with a valuation $v$ on the set $H$. We say that a complete Heyting algebra model \textit{$(H, \sqsubseteq, v)$ models a formula $\varphi$} if $\langle \varphi \rangle^H_v = \top \in H$, and we then write $(H, \sqsubseteq, v) \models \varphi$.  
\end{definition}

\begin{body}
In the remainder of the text, we normally use the symbol $\mathcal{H}$ to denote a complete Heyting algebra model, and unless otherwise noted, denote its underlying set as $H$, its partial order relation as $\sqsubseteq$, and its valuation as $v$.
\end{body}

\begin{definition}\label{def:semantic-validity}
Let $\Gamma$ denote a set of \textbf{LJ2}-formulas, and $\varphi$ denote one such formula. We write $\Gamma \vDash \varphi$, and call \textit{$\varphi$ a semantic consequence of $\Gamma$}, if the following holds: in every complete Heyting algebra model $\mathcal{H}$ so that $\mathcal{H} \models \psi$ for all $\psi \in \Gamma$, we also have $\mathcal{H} \models \varphi$.
\end{definition}

\begin{body}
The soundness of complete Heyting algebra semantics (Proposition~\ref{prop:soundness}) is well-known and was noted by Geuvers~\cite{geuvers-semantics} in 1994.
\end{body}

\begin{proposition}[Soundness]\label{prop:soundness}
Take a set of second-order propositional formulas $\Gamma$, and one such formula $\varphi$. If $\varphi$ is a syntactic consequence of $\Gamma$, then $\varphi$ is a semantic consequence of $\Gamma$ as well. 
\end{proposition}

\begin{body}
Throughout this text, we will use the unqualified term \textit{completeness} only for the assertion that for any set of second-order propositional formulas $\Gamma$, and one such formula $\varphi$, if $\varphi$ is a semantic consequence of $\Gamma$, then $\varphi$ is also a syntactic consequence of $\Gamma$. This is often referred to as \textit{strong completeness}, to distinguish it from the (generally weaker) statement, \textit{weak completeness}, that every formula valid in every model is derivable. See Section~\ref{sec:failure-of-completeness} for a more thorough background on completeness.
\end{body}

\begin{body}
Note that Chapter~12 of \textit{Lectures on the Curry-Howard isomorphism}~\cite{urzyczyn-lectures} states strong completeness of complete Heyting algebra semantics for second-order intuitionistic propositional logic, omitting the proof and citing Geuvers~\cite{geuvers-semantics}. However, Geuvers claims to prove only weak completeness in the cited article. Moreover, in later work, Zdanowski~\cite{zdanowski-quant} pointed out that the argument for weak completeness offered there contains a serious flaw. In Section~\ref{sec:failure-of-completeness} we show that strong completeness in fact fails for this semantics. As far as the author knows, the question of weak completeness has not been resolved, and weak completeness for complete Heyting algebra semantics remains unknown.
\end{body}

%%%%%%%%%%%%%%%%%%%%%%%%%%%%%%%%%%%%%%%%%%%%%%%%%%%%%%%%%%%%%%%%%%%%%%%%%%%%%%%%
\section{Bellissima's algebra}\label{sec:bellissima}

\begin{body}
Both of our applications make use of Bellissima's universal Kripke model~\cite{bellissima-finite} $K_2$ and the corresponding Heyting algebra of upsets $\mathfrak{H}_2$. Below, we briefly recall Bellissima's construction and relevant results about it, and clarify the relationship between $F_2$ and $\mathfrak{H}_2$ by defining the notion of propositionally representable elements.
\end{body}

\begin{definition}\label{def:bel-valuation}
A \textit{Bellissima valuation} is a subset of $\{a,b\}$.
\end{definition}

\begin{definition}\label{def:bel-model}
We define the \textit{set of Bellissima worlds at stage $n$}, denoted $K_{2,n}$ inductively on the stage $n$. Set
\begin{align*}
  & K_{2,-1} = \emptyset, & \\
  & K_{2,0} = \left\{  (\{a,b\}, \emptyset), (\{a\}, \emptyset), (\{b\}, \emptyset), (\emptyset, \emptyset)   \right\} &
\end{align*}
and for $n \geq 0$, let $K_{2,n+1}$ consist of the union of $K_{2,n}$ with the set of all pairs $(\beta, Y)$ where $\beta$ is a Bellissima valuation and $Y$ is a subset of $K_{2,n}$ so that $Y \setminus K_{2,n-1}$ is inhabited, and the following conditions hold for all $(\beta',Y') \in Y$:
\begin{enumerate}
  \item $Y' \subseteq Y$,
  \item $\beta \subseteq \beta'$,
  \item if $Y = Y' \cup \{(\beta', Y')\}$ then $\beta \subset \beta'$.
\end{enumerate}
We define the \textit{Bellissima model} $(K_2, \leq,\nu)$ to have as underlying set the union
$$K_2 = \bigcup_{n \in \mathbb{N}} K_{2,n}$$ with strict order relation $<$ given by
$$ (\beta, Y) < (\gamma, Z)  :\equiv (\gamma, Z) \in Y, $$
$\leq$ given by the reflexive closure of $<$, and valuation function given by $\nu(\beta,Y) = \beta$.
\end{definition}

\begin{body}
Notice that, having defined $\leq$, the three conditions of Definition~\ref{def:bel-model} guarantee that $\leq$ is transitive, that $\nu$ is monotone with respect to $\leq$, and that $K_2$ is reduced, i.e.~no two worlds have the same theory, respectively. 
\end{body}

\begin{proposition}\label{prop:bel-finite}
For each $n \in \mathbb{N}$, $K_{2,n}$ consists of finitely many elements. Moreover, the upset of any element $x \in K_2$ is finite.
\begin{proof}
Argue by induction. For the base case, we know that $K_{2,0}$ is finite. For the inductive case, if $K_{2,k}$ is finite then there are only finitely many ways to choose a subset $Y\subseteq K_{2,k}$, so $K_{2,k+1} \setminus K_{2,k}$ consists of at most finitely many elements. This proves that for each $n \in \mathbb{N}$, $K_{2,n}$ is finite. Since any $(\beta, V) \in K_2$ belongs to $K_{2,n+1}$ for some $n$, we get that $V \subseteq K_{2,n}$ is finite, and then $\uparrow (\beta, V) = V \cup \{ (\beta,V) \}$ is finite as required.
\end{proof}
\end{proposition}

\begin{definition}\label{def:bel-algebra}
The \textit{Bellissima algebra} $(\mathfrak{H}_2, \sqsubseteq)$ is the complete Heyting algebra formed by the upward-closed sets of $(K_2, \leq)$, ordered by inclusion. We define the map $\langle - \rangle : F_2 \rightarrow \mathfrak{H}_2$ over the class $[\varphi] \in F_2$ of the formula $\varphi$ as
$$ \langle [\varphi] \rangle := \SetComp{w \in K_2}{w \Vdash \varphi}.$$
The map is well-defined by soundness of Kripke semantics for intuitionistic propositional logic. We call elements of the form $\langle [\varphi] \rangle \in \mathfrak{H}_2$ \textit{propositionally representable}. 
\end{definition}

\begin{theorem}[Theorem~2.4~of~\cite{bellissima-finite}]\label{thm:bel-embedding}
The function $\langle - \rangle : F_2 \hookrightarrow \mathfrak{H}_2$ of Definition~\ref{def:bel-algebra} is a monomorphism of Heyting algebras.
\end{theorem}

\begin{theorem}[Corollary~2.8~of~\cite{bellissima-finite}]\label{thm:bellissima-principal-formula}
For any $w \in K_2$, one can find a formula $\psi_w$ of intuitionistic propositional logic in two free variables, so that $\langle [\psi_w] \rangle = \:\uparrow w$ in $\mathfrak{H}_2$.
\end{theorem}

\subsection{The $\nabla$ operator in the Bellissima algebra}

\begin{body}
In Proposition~\ref{prop:nabla-characterization-upsets} we prove that the coderivative admits a characterization in terms of the strict order $<$ on $K_2$. Note that this is just an instance of the topological characterization of the coderivative, specialized to Alexandroff spaces. As such, it holds even if $\mathfrak{H}_2$ is replaced with the Heyting algebra of upsets of any other partial order.
\end{body}

\begin{proposition}\label{prop:nabla-characterization-upsets}
The equality
$$\nabla (X) = \SetComp{w \in K_2}{\forall v > w. v \in X}$$
obtains for all $X \in \mathfrak{H}_2$.
\begin{proof}
First assume that $\forall v > w. v \in X$. We will prove that $w \in Q \sqcup (Q \Rightarrow X)$ holds for an arbitrary $Q \in \mathfrak{H}_2$. If $w \in Q$, then we are done. Otherwise, take any $x \geq w$ and assume $x \in Q$. Then $x \neq w$ since $w \not\in Q$, so $x > w$. The first assumption immediately lets us conclude that $x \in X$. This proves $\forall x \geq w. x \in Q \rightarrow x \in X$, the forcing condition for $Q \Rightarrow X$.

\vspace{0.5em} We prove the other inclusion by contrapositive. Assume $\exists v > w. v \not\in X$. We need to prove that $w \not\in Q \sqcup (Q \Rightarrow X)$ for some $Q \in \mathfrak{H}_2$. Take simply $Q = \uparrow v$. Then $w \not\in Q$ since $v > w$, so the first disjunct fails. The second disjunct fails too: if $w \in Q \Rightarrow X$, then for all $x \geq w$ with $x \in Q$, we have $x \in X$. Taking $x$ as $v$, we get a contradiction: $v \in Q$ and by assumption $v\not\in X$.
\end{proof}
\end{proposition}

\subsection{Rieger-Nishimura ladders in $K_2$}

\begin{body}
The Rieger-Nishimura ladder is the universal Kripke model for intuitionistic propositional logic in one variable. Its dual forms the Rieger-Nishimura lattice, the free Heyting algebra on one generator. Below, we fix a copy of the Rieger-Nishimura ladder inside $K_2$, which we use as a source of counterexamples in Sections~\ref{sec:xu-ye-theorem}~and~\ref{sec:failure-of-completeness}.
\end{body}

\begin{definition}\label{def:bel-rieger-nishimura}
We define a copy of the Rieger-Nishimura ladder in $K_2$ inductively as follows:
\begin{align*}
    & v_0 = (\{a,b\}, \emptyset), & \\
    & v_1 = (\{a\}, \emptyset), & \\
    & v_{n+2} = (\{a\}, \{ v_0, \dots, v_n\}). &
\end{align*}
A short induction argument shows that the conditions of Definition~\ref{def:bel-model} are satisfied, so $v_i$ belong to $K_2$. Note that $\{a\} \subseteq \nu(v_i)$, and thus $v_i \in \langle a \rangle$ for all $i \in \mathbb{N}$. Algebraically, this Rieger-Nishimura ladder corresponds to the substitution that sends the generator $a$ to $\top$.
\end{definition}

\begin{figure}[H]
\noindent\hfill
% https://q.uiver.app/#q=WzAsMTcsWzEsMCwie3ZfMX0iXSxbMiwwLCJ7dl8wfSJdLFsxLDEsInt2XzN9Il0sWzIsMSwie3ZfMn0iXSxbMSwyLCJ7dl81fSJdLFsyLDIsInt2XzR9Il0sWzEsMywie3ZfN30iXSxbMiwzLCJ7dl82fSJdLFsxLDQsIlxcdmRvdHMiXSxbMiw0LCJcXHZkb3RzIl0sWzIsNSwiXFx2ZG90cyJdLFszLDAsIlxcdGlueSB2XzA9KFxce2EsYlxcfSxcXGVtcHR5c2V0KSJdLFswLDAsIlxcdGlueSB2XzE9KFxce2FcXH0sXFxlbXB0eXNldCkiXSxbMywxLCJcXHRpbnkgdl8yPShcXHthXFx9LFxce3ZfMFxcfSkiXSxbMCwxLCJcXHRpbnkgdl8zPShcXHthXFx9LFxce3YwLHYxXFx9KSJdLFszLDIsIlxcdGlueSB2XzQ9KFxce2FcXH0sXFx7djAsdjEsdjJcXH0pIl0sWzAsMiwiXFx0aW55IHZfNT0oXFx7YVxcfSxcXHt2MCx2MSx2Mix2M1xcfSkiXSxbMywxXSxbNSwzXSxbNyw1XSxbNiw0XSxbOSw3LCIiLDAseyJzdHlsZSI6eyJib2R5Ijp7Im5hbWUiOiJkb3R0ZWQifX19XSxbOCw2LCIiLDAseyJzdHlsZSI6eyJib2R5Ijp7Im5hbWUiOiJkb3R0ZWQifX19XSxbNCwyXSxbMiwwXSxbMiwxXSxbNSwwXSxbNCwzXSxbNywyXSxbNiw1XSxbOCw3LCIiLDEseyJzdHlsZSI6eyJib2R5Ijp7Im5hbWUiOiJkb3R0ZWQifX19XSxbMTAsNiwiIiwxLHsic3R5bGUiOnsiYm9keSI6eyJuYW1lIjoiZG90dGVkIn19fV1d
\begin{tikzcd}
	{\tiny v_1=(\{a\},\emptyset)} & {{v_1}} & {{v_0}} & {\tiny v_0=(\{a,b\},\emptyset)} \\
	{\tiny v_3=(\{a\},\{v0,v1\})} & {{v_3}} & {{v_2}} & {\tiny v_2=(\{a\},\{v_0\})} \\
	{\tiny v_5=(\{a\},\{v0,v1,v2,v3\})} & {{v_5}} & {{v_4}} & {\tiny v_4=(\{a\},\{v0,v1,v2\})} \\
	& {{v_7}} & {{v_6}} \\
	& \vdots & \vdots \\
	&& \vdots
	\arrow[from=2-2, to=1-2]
	\arrow[from=2-2, to=1-3]
	\arrow[from=2-3, to=1-3]
	\arrow[from=3-2, to=2-2]
	\arrow[from=3-2, to=2-3]
	\arrow[from=3-3, to=1-2]
	\arrow[from=3-3, to=2-3]
	\arrow[from=4-2, to=3-2]
	\arrow[from=4-2, to=3-3]
	\arrow[from=4-3, to=2-2]
	\arrow[from=4-3, to=3-3]
	\arrow[dotted, from=5-2, to=4-2]
	\arrow[dotted, from=5-2, to=4-3]
	\arrow[dotted, from=5-3, to=4-3]
	\arrow[dotted, from=6-3, to=4-2]
\end{tikzcd}
\hfill\null
\caption{The Hasse diagram of the Rieger-Nishimura ladder constructed in Definition~\ref{def:bel-rieger-nishimura}. The arrows $x \rightarrow y$ represent the covering relation $x < y \wedge \neg \exists z. x < z \wedge z < y$.}
\label{fig:rieger-nishimura-ladder}
\end{figure}
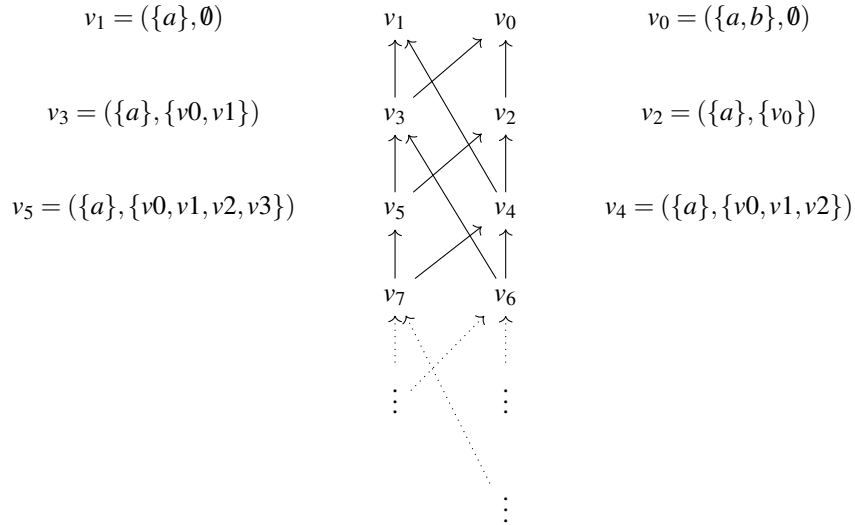

\section{Free Heyting algebras and topoi}\label{sec:xu-ye-theorem}

\begin{body}
In this section, we use the $\nabla$ operator to provide a simplified proof of the theorem, recently established by Xu~and~Ye~\cite{xu-ye-failure}, that the free Heyting algebra on two generators does not arise as the Heyting algebra of subterminal objects of any elementary topos.
\end{body}

\begin{body}
The development of this section does not use the semantics for second-order logic introduced in Section~\ref{sec:preliminaries}. However, it does make use of reasoning in the internal language of an elementary topos $\mathcal C$, where a formula in variable context $x_1:A_1,\ldots,x_n:A_n$ denotes a subobject of $A_1\times\cdots\times A_n$, and each sentence therefore denotes an arrow $1 \rightarrow \Omega$ (i.e.~a \textit{truth value}). In this semantics, quantifiers such as $\forall x:\Omega$ are given not by external quantification over all global elements $\mathcal{C}(1,\Omega)$, but by the appropriate adjoint to weakening. We will not introduce the internal language of topoi here, and refer the reader to Part~III~of~\cite{mclarty-topoi}.
\end{body}

\begin{body}
Much like the original argument of Xu~and~Ye~\cite{xu-ye-failure}, our argument uses Bellissima's universal Kripke model~\cite{bellissima-finite} $K_2$ and the corresponding Heyting algebra of upsets $\mathfrak{H}_2$ introduced above. However, in our argument, the point-free coderivative $\dera$ replaces their ad hoc obstruction, avoiding a calculation with three mutually recursive sequences and their six component internal encoding. This also eliminates the need to consider coprincipal Bellissima formulas and lets us use Bellissima worlds instead of Bellissima stages in a later part of the proof.
\end{body}

\subsection{The non-representability of $\dera$}

\begin{body}
We give a sufficient condition for propositional nonrepresentability in Lemma~\ref{lem:representables-have-finitely-many-minimal-pts}. To prove it, we recall two further results from the literature, as Theorems~\ref{thm:urquhart-free}~and~\ref{thm:darniere-irreducible} below. We note that these exact theorems, along with the two theorems recalled in Section~\ref{sec:bellissima}, are also the prerequisites of the original proof of Theorem~\ref{theorem:xu-ye} given by Xu~and~Ye.
\end{body}

\begin{theorem}[Theorem~3~of~\cite{urquhart-free}]\label{thm:urquhart-free}
Every element of a free Heyting algebra can be written as a finite join of join-irreducible elements.
\end{theorem}

\begin{theorem}[Corollary~7.2~of~\cite{darniere-fgfha}]\label{thm:darniere-irreducible}
If $[\varphi]\in F_2$ is join-irreducible, then so is its image $\langle [\varphi] \rangle \in \mathfrak{H}_2$.
\end{theorem}

\begin{lemma}\label{lem:representables-have-finitely-many-minimal-pts}
Every propositionally representable element of $\mathfrak{H}_2$ has only finitely many $\leq$-minimal points.
\begin{proof}
By Theorems~\ref{thm:urquhart-free}~and~\ref{thm:darniere-irreducible}, every propositionally representable element of $\mathfrak{H}_2$ can be written as a finite join of join-irreducibles $$X = Y_1 \sqcup \dots \sqcup Y_n.$$ % we use that <-> is a homomorphism too, but that can be supressed
Every minimal element $w \in X$ has to be a minimal element of one of these $Y_i$. But in an algebra of upsets, every join-irreducible element can have at most one minimal point\footnote{This is generally true. One can write $Y = Y \setminus \{w\} \cup Y \setminus \{m\}$ where $m$ denotes another minimal element.}
Thus, by the pigeonhole principle, the set $X$ contains at most $n$ minimal points.
\end{proof}
\end{lemma}

\begin{proposition}\label{prop:nabla-a-not-representable}
The set $\dera \in \mathfrak{H}_2$ is not propositionally representable.
\begin{proof}
We prove that each point of the form $(\emptyset,\{v_0,\dots,v_n\})$ is minimal in $\dera$. We can construct infinitely many points of this form, so by Lemma~\ref{lem:representables-have-finitely-many-minimal-pts}, we obtain that $\dera$ is not propositionally representable.

\vspace{0.5em} By the characterization of Proposition~\ref{prop:nabla-characterization-upsets}, we have $(\emptyset,\{v_0,\dots,v_n\}) \in \dera$ since the $v_i$ all belong to $\langle a \rangle$ by construction. However, we do not have $(\emptyset,\{v_0,\dots,v_n\}) \in \langle a \rangle$. Using Proposition~\ref{prop:nabla-characterization-upsets} once again, we see that this makes $(\emptyset,\{v_0,\dots,v_n\})$ a minimal element of $\dera$: any $x < (\emptyset,\{v_0,\dots,v_n\})$ by definition lies below an element outside of $\langle a \rangle$.
\end{proof}
\end{proposition}

\subsection{Representability in the internal language of a topos}

\begin{body}
From here onward, we assume for a contradiction that we work in a topos $\mathcal{C}$ with terminal object $1$ and subobject classifier $\Omega$, whose Heyting algebra of subterminal objects $\mathcal{C}(1,\Omega)$ is isomorphic to $F_2$. We will suppress this isomorphism, and freely identify arrows $\mathcal{C}(1,\Omega)$ with elements of $F_2$. In particular, we identify the free generators with arrows $a \in \mathcal{C}(1, \Omega)$ and $b \in \mathcal{C}(1, \Omega)$. Throughout the text, we use the notation $\wedge, \vee, \rightarrow$ for the internal Heyting algebra operations on $\Omega$.
\end{body}

\begin{body}
We will construct, in the internal language of $\mathcal{C}$, a unary predicate $\mathrm{Gen}(-)$ so that the truth value $\tau$ of the internal logic sentence
$$ \forall x : \Omega.\  \mathrm{Gen}(x)\rightarrow (x \vee (x \rightarrow a)) $$
satisfies $\langle \tau \rangle = \dera$. Since the latter is not propositionally representable, this gives the required contradiction to conclude that the topos $\mathcal{C}$ cannot exist.
\end{body}

\begin{definition}[in $\mathcal{C}$]
Consider $S : \mathcal{P}(\Omega)$. We call $S$ a \textit{Heyting predicate containing} $a,b$, and write $\mathrm{Hp}(S)$ if the following all hold:
\begin{enumerate}
  \item $S(\bot), S(\top), S(a), S(b)$,
  \item $\forall x,y : \Omega.\  S(x) \rightarrow S(y) \rightarrow S(x \wedge y)$,
  \item $\forall x,y : \Omega.\  S(x) \rightarrow S(y) \rightarrow S(x \vee y)$, and
  \item $\forall x,y : \Omega.\  S(x) \rightarrow S(y) \rightarrow S(x \rightarrow y)$.
\end{enumerate}
We say that $x : \Omega$ \textit{is generated by} $a,b$, and write $\mathrm{Gen}(x)$, if $x$ satisfies every Heyting predicate generated by $a,b$, in other words if $\forall S: \mathcal{P}(\Omega).\  \mathrm{Hp}(S) \rightarrow S(x)$ holds.
\end{definition}

\begin{proposition}\label{prop:external-prop-generated}
Let $\varphi$ denote a formula of intuitionistic propositional logic in two propositional variables $a,b$. Then $\vdash_\mathcal{C} \mathrm{Gen}(\varphi) : \Omega$ holds in the internal logic of $\mathcal{C}$.
\begin{proof}
Induction on the structure of the formula. The base cases of $\bot,\top,a,b$ follow from the first clause in the definition of $\mathrm{Hp}$ and the inductive cases follow from the other three clauses.
\end{proof}
\end{proposition}

\begin{body}
Note that when $\mathcal{C}$ has a natural numbers object, even the internal analogue of Proposition~\ref{prop:external-prop-generated} holds. We do not require our elementary topoi to have a natural numbers object, and the external version suffices for our purposes.
\end{body}

\begin{lemma}\label{lem:finale-1}
Let $\tau \in \mathcal{C}(1,\Omega)$ denote the truth value of the internal logic sentence
$$ \forall x : \Omega.\  \mathrm{Gen}(x)\rightarrow (x \vee (x \rightarrow a)).$$
Then $\langle \tau \rangle \sqsubseteq \dera$.
\begin{proof}
We prove the contrapositive, that for a world $w \in K_2$, if $w \not\in \dera$ then $w \not\in \langle \tau \rangle$. Take such a world $w \not\in \dera$ and assume for a contradiction that $w \in \langle\tau\rangle$. By the characterization of Proposition~\ref{prop:nabla-characterization-upsets}, we can find $v > w$ in $K_2$ so that $v \not \in \langle a \rangle$. In turn, by Proposition~\ref{prop:external-prop-generated}, we have that $\vdash_\mathcal{C} \mathrm{Gen}(\psi_v) : \Omega$ and hence that
$$ \vdash_\mathcal{C} \left(\forall x: \Omega.\ \mathrm{Gen}(x) \rightarrow (x \vee (x \rightarrow a))\right) \rightarrow (\psi_v \vee (\psi_v \rightarrow a)) : \Omega$$
holds in the internal logic of $\mathcal{C}$. Externally, it equals the arrow $\top \in \mathcal{C}(1,\Omega)$. Since $x \Rightarrow y = \top$ precisely if $x \sqsubseteq y$ in any Heyting algebra, passing the resulting inequality through the map $\langle-\rangle$ yields
$$ \langle \tau \rangle \sqsubseteq \langle[\psi_v \vee (\psi_v \rightarrow a) ]\rangle $$
and so
$$ \langle \tau \rangle \sqsubseteq \langle[\psi_v]\rangle \sqcup (\langle[\psi_v]\rangle \Rightarrow \langle a\rangle) = \uparrow v \sqcup (\uparrow v \Rightarrow \langle a\rangle)$$
holds in $\mathfrak{H}_2$. It suffices to show that $w \not\in\: \uparrow v \sqcup (\uparrow v \Rightarrow \langle a\rangle)$. Clearly $w \not\in\: \uparrow v$. But we cannot have $w \in\: \uparrow v \Rightarrow \langle a\rangle$ either, since $v$ itself gives a counterexample, a $v \geq w$ so that $v \in\: \uparrow v$ but $v \not\in \langle a \rangle$.
\end{proof}
\end{lemma}

\subsection{Logic over a Bellissima world}

\begin{body}
In the remainder of this section we will show that $\dera \sqsubseteq \langle \tau \rangle$ as well. We use an equivalence relation $\sim_w$ that is largely analogous to the $\equiv_d$ in the proof of Xu~and~Ye~\cite{xu-ye-failure}, but works directly over a finite principal upset instead of an entire stage $K_{2,d}$ of the universal Kripke model. % NB prop:logic-over-a-world strengthens their observation in sec4, which concerns one orbit of $u_\bullet$ required for their proof. This is the periodicity that 
\end{body}

\begin{definition}\label{def:equivalence-over-a-world}
Take any $w \in K_2$. We define the equivalence relation $\varphi_1 \sim_w \varphi_2$ on formulas of intuitionistic propositional logic in the two free variables $a,b$ so that it holds precisely if
$$ \vdash \psi_w \rightarrow (\varphi_1 \leftrightarrow \varphi_2) $$
is derivable in intuitionistic propositional logic.
\end{definition}

\begin{proposition}\label{prop:logic-over-a-world}
There are only finitely many equivalence classes of formulas under the relation $\sim_w$.
\begin{proof}
Using the fact that $X \rightarrow (Y \leftrightarrow Z)$ and $(X \wedge Y) \leftrightarrow (X \wedge Z)$ are equivalent in intuitionistic logic, we get that $\varphi_1 \sim_w \varphi_2$ precisely if $[\psi_w \wedge \varphi_1] = [\psi_w \wedge \varphi_2]$. Using the fact that $\langle-\rangle:F_2 \hookrightarrow \mathfrak{H}_2$ is a monomorphism, we get that this holds precisely if $\uparrow w \sqcap \langle[\varphi_1]\rangle =\: \uparrow w \sqcap \langle[\varphi_2]\rangle$ in $\mathfrak{H}_2$. The upset of every $w \in K_2$ is finite by Proposition~\ref{prop:bel-finite}, and hence the downset of every $\uparrow w \in \mathfrak{H}_2$ is finite as well, proving the claim.
\end{proof}
\end{proposition}

\begin{body}
Fix upfront some $w \in K_2$ and a finite sequence of formulas $i_1, i_2, \dots, i_n$ so that $\{i_1,i_2,\dots,i_n\}$ exhaust all of the (by Proposition~\ref{prop:logic-over-a-world} finitely many) equivalence classes of formulas under $\sim_w$. Notice that if $\varphi_1 \sim_w \varphi_2$ for intuitionistic propositional formulas $\varphi_1$ and $\varphi_2$, then by definition there is a derivation $\vdash \psi_w \rightarrow (\varphi_1 \leftrightarrow \varphi_2)$ in intuitionistic propositional logic. One can \textit{replay} this derivation in the internal logic of $\mathcal{C}$ as well, obtaining a derivation of
$ \vdash_\mathcal{C} \psi_w \rightarrow (\varphi_1 \leftrightarrow \varphi_2) : \Omega$
internally. We will make use of this fact in Propositions~\ref{prop:internal-interval-j-is-hpg}, \ref{prop:internal-gen-is-interval-j}~and~\ref{prop:internal-nabla}, which hold in the internal logic of $\mathcal{C}$.
\end{body}

\begin{definition}[in $\mathcal{C}$]\label{def:internal-interval-j}
Define the predicate $J_w(-) : \mathcal{P}(\Omega)$ as
$$ J_w(x) := \bigvee_{k=1}^n (\psi_w \rightarrow (x \leftrightarrow i_k)).$$
\end{definition}

\begin{body}
Note that this predicate has the same form as $I_\bullet$ in the proof of Xu~and~Ye~\cite{xu-ye-failure}, with our $\sim_w$ substituted for their $\equiv_d$.
\end{body}

\begin{proposition}[in $\mathcal{C}$]\label{prop:internal-interval-j-is-hpg}
The predicate $J_w(-) : \mathcal{P}(\Omega)$ satisfies $\mathrm{Hp}(J_w)$.
\begin{proof}
The base cases are immediate. We prove closure under disjunction: the other closure proofs are very similar. We have to show
$$ \forall x, y : \Omega.\: J_w(x) \rightarrow J_w(y) \rightarrow J_w(x \vee y).$$
Expanding the definition of $J_w(x \vee y)$, this means
$$ \forall x, y : \Omega.\: J_w(x) \rightarrow J_w(y) \rightarrow \bigvee_{k=1}^n (\psi_w \rightarrow ((x \vee y) \leftrightarrow i_k)).$$
So take $x,y : \Omega$ and assume $J_w(x)$ and $J_w(y)$. Use disjunction elimination on $J_w(x)$ to obtain some $i_p$ such that $\psi_w \rightarrow (x \leftrightarrow i_p)$, and similarly $i_q$ so that $\psi_w \rightarrow (y \leftrightarrow i_q)$. By Proposition~\ref{prop:logic-over-a-world}, we know some $i_r$ such that $i_p \vee i_q \sim_w i_r$. We prove the disjunct $\psi_w \rightarrow ((x \vee y)\leftrightarrow i_r)$.

\vspace{0.5em} Assume $\psi_w$. We then need to show $(x \vee y) \leftrightarrow i_r$. But from $\psi_w \rightarrow (x \leftrightarrow i_p)$ we have $x \leftrightarrow i_p$, and from $\psi_w \rightarrow (y \leftrightarrow i_q)$ we have $y \leftrightarrow i_q$. Thus, $(x \vee y) \leftrightarrow (i_p \vee i_q)$. But we also know that $\psi_w \rightarrow ((i_p \vee i_q) \leftrightarrow i_r)$, so by the assumption $(i_p \vee i_q) \leftrightarrow i_r$. Transitivity of $\leftrightarrow$ then gives $(x \vee y) \leftrightarrow i_r$ as required. 
\end{proof}
\end{proposition}

\begin{proposition}[in $\mathcal{C}$]\label{prop:internal-gen-is-interval-j}
We have $\forall x:\Omega.\: \mathrm{Gen}(x) \rightarrow J_w(x)$.
\begin{proof}
Take any $x: \Omega$ and assume $\mathrm{Gen}(x)$. By definition of $\mathrm{Gen}(x)$, this means that every $S:\mathcal{P}(\Omega)$ so that $\mathrm{Hp}(S)$ holds for $x$. But by Proposition~\ref{prop:internal-interval-j-is-hpg}, $\mathrm{Hp}(J_w)$ holds. Hence, $J_w(x)$ holds as well.
\end{proof}
\end{proposition}

\begin{body}
For the internal argument of Proposition~\ref{prop:internal-nabla}, further assume that $w \in \dera$. Then we have $\varphi \vee (\varphi \rightarrow a) \sim_w \top$ for any formula $\varphi$ of intuitionistic logic in two free variables $a,b$, and in particular this holds for any $i_p$ in the sequence $i_1,\dots,i_n$. As before, we can replay the corresponding proof in the internal logic of our topos $\mathcal{C}$, and obtain $\vdash_\mathcal{C} \psi_w \rightarrow (i_p \vee (i_p \rightarrow a)):\Omega$.
\end{body}

\begin{proposition}[in $\mathcal{C}$]\label{prop:internal-nabla}
We have $\psi_w \rightarrow \forall x : \Omega.\: \mathrm{Gen}(x) \rightarrow (x \vee (x \rightarrow a))$.
\begin{proof}
Assume $\psi_w$. Take some $x : \Omega$ and assume $\mathrm{Gen}(x)$. Then by Proposition~\ref{prop:internal-gen-is-interval-j}, we obtain $J_w(x)$. Use disjunction elimination on $J_w(x)$ to obtain a disjunct $\psi_w \rightarrow (x \leftrightarrow i_p)$ for some $i_p$. From the first assumption, we get $x \leftrightarrow i_p$. But $\psi_w \rightarrow (i_p \vee (i_p \rightarrow a))$ holds, so $i_p \vee (i_p \rightarrow a)$ holds as well. Replacing equivalent formulas with equivalent formulas, we obtain $x \vee (x \rightarrow a)$ as claimed.
\end{proof}
\end{proposition}

\begin{lemma}\label{lem:finale-2}
Let $\tau \in \mathcal{C}(1,\Omega)$ denote the truth value of the internal logic sentence
$$ \forall x : \Omega.\  \mathrm{Gen}(x)\rightarrow (x \vee (x \rightarrow a)).$$
Then $\dera \sqsubseteq \langle \tau \rangle$.
\begin{proof}
Assume that $w \in \dera$. We claim that then $w \in \langle \tau\rangle$. From the internal Proposition~\ref{prop:internal-nabla} we know that
$$ \vdash_\mathcal{C} \psi_w \rightarrow \forall x: \Omega.\: \mathrm{Gen}(x) \rightarrow (x \vee (x \rightarrow a)) : \Omega$$
and hence that $[\psi_w] \sqsubseteq \tau$ in $F_2$. The latter, passed through the function $\langle - \rangle$ immediately gives $\uparrow w = \langle[\psi_w]\rangle \sqsubseteq \langle \tau \rangle$ in $\mathfrak{H}_2$, and hence $w \in \langle \tau \rangle$ as claimed.
\end{proof}
\end{lemma}

\begin{theorem}[Xu-Ye~\cite{xu-ye-failure}]\label{theorem:xu-ye}
The free Heyting algebra on two generators does not occur as the Heyting algebra of subterminal objects of any elementary topos.
\begin{proof}
By Lemma~\ref{lem:finale-1} and Lemma~\ref{lem:finale-2}, under the isomorphism $\mathcal{C}(1, \Omega) \cong F_2$, the truth value $\tau$ of the internal sentence
$$\forall x: \Omega. \mathrm{Gen}(x) \rightarrow (x \vee (x \rightarrow a))$$
would be an element of $F_2$ that propositionally represents $\dera$. By Proposition~\ref{prop:nabla-a-not-representable}, no such element exists.
\end{proof}
\end{theorem}

%%%%%%%%%%%%%%%%%%%%%%%%%%%%%%%%%%%%%%%%%%%%%%%%%%%%%%%%%%%%%%%%%%%%%%%%%%%%%%%%

\section{Complete Heyting algebras are not complete}\label{sec:failure-of-completeness}

\begin{body}
In this section we prove that complete Heyting algebra semantics does not give a strongly complete semantics for second-order intuitionistic propositional logic. We do this explicitly, by writing down a theory $\mathcal{T}$ and formula $\varphi$ such that $\mathcal T \vDash \varphi$, but $\mathcal{T} \not\vdash_2 \varphi$. We present a semantic proof, chiefly so that we can reuse Bellissima's algebra as a counterexample. At the end of the section, we briefly sketch a proof-theoretic argument which proves the same result. The main technical tool in both approaches is a second-order definition of the least fixed point of the coderivative (Proposition~\ref{prop:nabla-least-fixed-point}).
\end{body}

\subsection{Fixed points of the coderivative}

\begin{body}
Let $\varphi$ denote an intuitionistic second-order propositional formula. For the remainder of this section, $\Delta(\varphi)$ abbreviates the formula
$$ \forall Y. Y \vee (Y \rightarrow \varphi)$$
for some variable $Y$ fresh in $\varphi$.
\end{body}

\begin{body}
Note that if $(H,\sqsubseteq)$ is a complete Heyting algebra and $v$ is a valuation with values in $H$, then $\langle \Delta(\varphi) \rangle^H_v = \nabla \langle \varphi \rangle^H_v$.
%\begin{align*}
%  \langle \Delta(\varphi) \rangle^H_v =& \\
%  =& \langle \forall Y. Y \vee (Y \rightarrow \varphi) \rangle^H_v \\
%  =& \bigsqcap_{y \in H} \langle Y \vee (Y \rightarrow \varphi) \rangle^H_{v(Y\mapsto y)} \\
%  =& \bigsqcap_{y \in H} \langle Y \rangle^H_{v(Y\mapsto y)} \vee \langle Y \rightarrow \varphi \rangle^H_{v(Y\mapsto y)} \\
%  =& \bigsqcap_{y \in H} y \vee \left(\langle Y \rangle^H_{v(Y\mapsto y)} \rightarrow  \langle \varphi \rangle^H_{v(Y\mapsto y)}\right) \\
%  =& \bigsqcap_{y \in H} y \vee \left(y \rightarrow  \langle \varphi \rangle^H_{v(Y\mapsto y)}\right) \\
%  =& \bigsqcap_{y \in H} y \vee \left(y \rightarrow  \langle \varphi \rangle^H_{v}\right)
%\end{align*}
%where the last equality uses the fact that $Y$ does not occur free in $\varphi$. Put another way, $\langle \Delta(\varphi) \rangle^H_v = \nabla \langle \varphi \rangle^H_v$.
\end{body}

\begin{proposition}\label{prop:nabla-monotone}
The map $\nabla: H \rightarrow H$ is order-preserving in any complete Heyting algebra $(H, \sqsubseteq)$.
\begin{proof}
By semantic soundness (Proposition~\ref{prop:soundness}) it suffices to prove that $$\forall X. \forall Y. (X \rightarrow Y) \rightarrow \Delta(X) \rightarrow \Delta(Y)$$ holds internally in second-order intuitionistic propositional logic. So take $X,Y$, assume $X \rightarrow Y$ and $\Delta(X)$. Take an arbitrary $Q$. We will prove $Q \vee (Q \rightarrow Y)$. By $\Delta(X)$, we know that in particular $Q \vee (Q \rightarrow X)$ must hold. We analyze the two possibilities. If $Q$ holds, we are done since the first disjunct of $Q \vee (Q \rightarrow Y)$ holds. Otherwise, $Q \rightarrow X$ holds, and since $X \rightarrow Y$, by transitivity of implication the second disjunct of $Q \vee (Q \rightarrow Y)$ holds instead.
\end{proof}
\end{proposition}

\begin{body}
By the Knaster-Tarski fixed point theorem, the operator $\nabla$ has a $\sqsubseteq$-least fixed point in any complete Heyting algebra $(H, \sqsubseteq)$. In Proposition~\ref{prop:nabla-least-fixed-point}, using a trick of Wadler~\cite{wadler-recforfree}, we show that the least fixed point of $\nabla$ is definable already in second-order propositional logic. Since $X \sqsubseteq \nabla X$ (because $X \sqsubseteq Y \Rightarrow X$), constructing an $X$ so that $\nabla(X) \sqsubseteq X$ immediately yields a fixed point.
\end{body}

\begin{proposition}\label{prop:nabla-least-fixed-point}
Take a complete Heyting algebra $(H, \sqsubseteq)$. Let $M_H$ denote the interpretation
$$ \langle \forall X. (\Delta(X) \rightarrow X) \rightarrow X \rangle^H_v$$
of the sentence $M:=\forall X. (\Delta(X) \rightarrow X) \rightarrow X$ of second-order propositional logic. Then $\nabla(M_H) = M_H$. Moreover, for any $x \in H$ so that $\nabla(x) \sqsubseteq x$, we have $M_H \sqsubseteq x$ as well.
\begin{proof}
As in Proposition~\ref{prop:nabla-monotone}, it will suffice to prove the two statements
$$\Delta(M) \rightarrow M \text{ and } \forall X. (\Delta(X) \rightarrow X) \rightarrow M \rightarrow X$$
internally to intuitionistic second-order propositional logic. In fact, the same argument works for any definable map whose monotonicity is provable in \textbf{LJ2} (see~\cite{wadler-recforfree}). First, we prove
$$ \forall X. (\Delta(X) \rightarrow X) \rightarrow M \rightarrow X. $$
Take arbitrary $X$, and assume that both $\Delta(X) \rightarrow X$ and $M$ hold. We will show that $X$ holds as well. Since $M$ holds, we have $(\Delta X \rightarrow X) \rightarrow X$. But by assumption, $\Delta(X) \rightarrow X$ holds. Thus, $X$ holds as claimed. Second, we prove $\Delta(M) \rightarrow M$, in other words
$\Delta(M) \rightarrow \forall X. (\Delta(X) \rightarrow X) \rightarrow X.$
Assume $\Delta(M)$ and take arbitrary $X$ so that $\Delta(X) \rightarrow X$. We claim that $X$ holds. Since $\Delta(X) \rightarrow X$, the previous result gives $M \rightarrow X$. But then by the argument of Proposition~\ref{prop:nabla-monotone}, $\Delta(M) \rightarrow \Delta(X)$. Since $\Delta(M)$ is an assumption, $\Delta(X)$ holds, and then by $\Delta(X) \rightarrow X$ so does $X$ as claimed.
\end{proof}
\end{proposition}

\begin{body}
Note that topologically, the least fixed point of $\nabla$ constructed in Proposition~\ref{prop:nabla-least-fixed-point} computes the open complement of the \textit{perfect kernel}, the largest closed subset of that space which contains no isolated points.
\end{body}

\subsection{$\nabla$-sequences}

\begin{definition}\label{def:nabla-sequence}
Let $N$ denote an index set, $s: N \rightarrow N$ an endomorphism of $N$, and $(H, \sqsubseteq)$ some complete Heyting algebra. We call an $N$-indexed set $P: N \rightarrow H$ a \textit{$\nabla$-sequence} if $\nabla P_{s(n)} \sqsubseteq P_n$ for all $n \in N$. We call the $\nabla$-sequence \textit{trivial} if $P_n = \top_H$ holds for all $n \in N$.
\end{definition}

\begin{lemma}\label{lem:nabla-sequence-triviality}
In a complete Heyting algebra model $\Hh$ so that $\Hh \models M$, every $\nabla$-sequence is trivial.
\begin{proof}
Take any $\nabla$-sequence $P: N \rightarrow H$ and denote the corresponding endomorphism $s: N \rightarrow N$. The greatest lower bound $\bigsqcap_{x \in N} P_x$ exists in $H$ by virtue of completeness. We first prove that $\bigsqcap_{x \in N} P_x$ is a fixed point of $\nabla$. Set $\rho = \bigsqcap_{x \in N} P_x$.

\vspace{0.5em}\textbf{Claim.} We claim that $\nabla \rho \sqsubseteq \bigsqcap_{x\in N} \nabla (P_{s(x)})$. It is enough to show that $\nabla \rho \sqsubseteq \nabla P_{s(i)}$ for any $i \in N$. In turn, by Proposition~\ref{prop:nabla-monotone}, it suffices to show $\rho \sqsubseteq P_{s(i)}$. But the latter is immediate from the definition of $\rho$.

\vspace{0.5em}Since $\nabla P_{s(x)} \sqsubseteq P_x$ for all $x \in N$, our claim gives
$$ \nabla \rho \sqsubseteq \bigsqcap_{x\in N} \nabla (P_{s(x)}) \sqsubseteq \bigsqcap_{x\in N} (P_x) = \rho.$$
This means that $\rho$ is a fixed point of $\nabla$. But by Proposition~\ref{prop:nabla-least-fixed-point}, the interpretation $\langle M \rangle^H_v = M_H$ gives the \textit{least} such fixed point, so $M_H \sqsubseteq \rho$. Thus, since $\Hh \models M$, we have
$$\top_H = M_H \sqsubseteq \rho \sqsubseteq P_x$$
for any $x \in N$, in other words the $\nabla$-sequence $P$ is trivial.
\end{proof}
\end{lemma}

\subsection{The theory of a free $\nabla$-sequence}

\begin{body}
We now introduce the ``theory of a free $\nabla$-sequence'', a set $\mathcal{T}$ of second-order propositional formulas which essentially asserts that the free propositional variables $P_0, P_1, P_2,\dots$ which occur in $\mathcal{T}$ form a $\nabla$-sequence with index set $\mathbb{N}$ and endomorphism $s(x) = x + 1$.
\end{body}

\begin{definition}\label{def:theory-of-nabla-seq}
The set $\mathcal{T}$ consists of the following \textbf{LJ2}-formulas:
\begin{itemize}
    \item the sentence $\forall X. (\Delta(X) \rightarrow X) \rightarrow X$, and
    \item the formulas $\Delta P_{n+1} \rightarrow P_n$ for each $n \in \mathbb{N}$.
\end{itemize}
We call $\mathcal{T}$ the \textit{theory of a free $\nabla$-sequence}.
\end{definition}

\begin{lemma}\label{lem:free-nabla-semantically-valid}
In complete Heyting algebra semantics, $\mathcal{T} \vDash P_0$.
\begin{proof}
Take an arbitrary complete Heyting algebra model $\Hh$ so that $\Hh \models \varphi$ for each formula $\varphi \in \mathcal{T}$. Then in particular $\Hh \models \forall X. (\Delta(X) \rightarrow X) \rightarrow X$, and the map $Q : \mathbb{N} \rightarrow H$ given by $Q_i = \langle P_i \rangle^H_v$ gives a $\nabla$-sequence with endomorphism map $s(x) = x + 1$ in $\Hh$. Consequently, by Lemma~\ref{lem:nabla-sequence-triviality},  $\langle P_i \rangle^H_v = Q_i= \top_H$ and therefore $\Hh \models P_i$ for every $i \in \mathbb{N}$. Taking $i=0$, we obtain that $\Hh \models P_0$. Since this holds for an arbitrary $\Hh$, we conclude $\mathcal{T} \vDash P_0$.
\end{proof}
\end{lemma}

\begin{lemma}\label{lem:free-nabla-not-derivable}
The free $\nabla$-sequence is nontrivial: in other words, $\mathcal{T} \not\vdash_2 P_0$.
\begin{proof}
Assume for a contradiction that $\mathcal{T} \vdash_2 P_0$. By definition of syntactic consequence, this means that there is a finite sequence $\Lambda$ consisting of axioms from $\mathcal{T}$ so that $\Lambda \vdash_2 P_0$. We will construct a valuation $v$ so that the complete Heyting algebra model $\mathcal{H} = (\mathfrak{H}_2, \sqsubseteq, v)$ satisfies $\mathcal{H} \models \lambda$ for each $\lambda \in \Lambda$, but for which $\mathcal{H} \not\models P_0$. We use $\mathfrak{H}_2$ for economy of presentation, but note that many simpler algebras would also suffice to construct the desired counterexample.

\vspace{0.5em} Let $N\in \mathbb{N}$ denote the smallest natural number so that the truncated theory
$$ \mathcal{T}_N = \{M \} \cup \SetComp {\Delta P_{n+1} \rightarrow P_n}{n < N}$$
contains all the formulas in $\Lambda$. Take any valuation satisfying $v(P_i) = \nabla^{N-i} (\emptyset)$ for every $i \leq N$, where $\nabla^k(-)$ denotes iterated application of $\nabla(-)$ exactly $k$ times. Then we have
\begin{align*}
 \langle \Delta P_{n+1} \rightarrow P_n \rangle^{\mathfrak{H}_2}_v &
 = \nabla v(P_{n+1}) \Rightarrow v(P_n) \\
 & = \nabla \nabla^{N-n-1}(\emptyset) \Rightarrow \nabla^{N-n}(\emptyset) \\
 & = \nabla^{N-n}(\emptyset) \Rightarrow \nabla^{N-n}(\emptyset) \\ 
 & = K_2
\end{align*}
for all $n < N$ as well.
However, $\langle P_0 \rangle^{\mathfrak{H}_2}_v = \nabla^{N} (\emptyset)$. We need to show that $\nabla^{N} (\emptyset)$ is not $\top$ in $\mathfrak{H}_2$. This amounts to finding a single world so that $w \not\in \nabla^{N} \bot$. But that is straightforward: by Proposition~\ref{prop:nabla-characterization-upsets}, $w \not\in \nabla X$ precisely if $\exists x > w. x \not\in X$. Consequently, since $x \not\in \bot$ holds for any world $x$, we have $w \not\in \nabla^N (\emptyset)$ precisely if there is a chain
$$ w < x_0 < \dots < x_{N-1}$$
starting from $w$. But e.g. the copy of the Rieger-Nishimura ladder obtained in Definition~\ref{def:bel-rieger-nishimura} already contains worlds with arbitrarily long chains above them (see Figure~\ref{fig:rieger-nishimura-ladder}). Taking one such element $w$ with a chain containing at least $N$ strict inequalities above it, we get that $w \not\in \nabla^N (\emptyset)$, and hence $\nabla^N (\emptyset) \neq K_2$.

\vspace{0.5em} All that's left to prove is that $\mathcal{H} \models M$ itself, in other words that $\langle M \rangle^{\mathfrak{H}_2}_v = K_2$. From Proposition~\ref{prop:nabla-least-fixed-point}, we know that the interpretation of $M$ is the least fixed point of $\nabla$ in $\mathfrak{H}_2$. We will show that this least fixed point is $K_2$ itself. Assume $\nabla(X) = X$ for some $X \in \mathfrak{H}_2$. We prove that each world $w \in K_{2,n}$ belongs to $X$ by induction on $n$. In the base case, the elements of $K_{2,0}$ do not have any elements above them, so vacuously belong to $\nabla(X)$ by Proposition~\ref{prop:nabla-characterization-upsets}. Since $\nabla (X) = X$, they belong to $X$ as well. In the inductive step, the elements of $K_{2,k}$ belong to $X$ by assumption. If $w \in K_{2,k+1}$ and $x > w$, then $x \in K_{2,k}$. This means that every chain starting from an element of $K_{2,k+1}$ enters $K_{2,k}$ immediately, all elements of $K_{2,k+1}$ belong to $\nabla(X)$, and by $\nabla (X) = X$, to $X$ as well. The induction thus proves that $X = K_2$ for any fixed point $X \in \mathfrak{H}_2$ of the $\nabla$ operator. Since $\langle M \rangle^{\mathfrak{H}_2}_v$ is a fixed point of $\nabla$, we get $\langle M \rangle^{\mathfrak{H}_2}_v = K_2$, in other words $\mathcal{H} \models M$.

\vspace{0.5em} Under the assumption $\mathcal{T} \vdash P_o$, we proved that $\Lambda \vdash_2 P_0$, and that there exists some complete Heyting algebra model $\mathcal{H}$ such that $\mathcal{H} \models \lambda$ for every $\lambda \in \Lambda$, but $\mathcal{H} \not\models P_0$, thereby contradicting soundness (Proposition~\ref{prop:soundness}).
\end{proof}
\end{lemma}

\begin{theorem}\label{thm:cha-not-complete}
Complete Heyting algebra semantics is not strongly complete for intuitionistic second-order propositional logic.
\begin{proof}
Immediate by Lemma~\ref{lem:free-nabla-semantically-valid} and Lemma~\ref{lem:free-nabla-not-derivable}.
\end{proof}
\end{theorem}

\subsection{Sketch of a proof-theoretic argument}

\begin{body}
The semantics-heavy proof of Lemma~\ref{lem:free-nabla-not-derivable} admits a purely proof-theoretic counterpart, via a proof translation of \textbf{LJ2} into the modal logic \textbf{KM} of Kuznetsov and Muravitsky~\cite{muravitsky-logic-km}. Litak~\cite{litak-modalities} notes that the operator $\Delta$ acts like the modal operator $\square$ of \textbf{KM}, and the usual Kripke semantics of the logic \textbf{KM} is known not to satisfy strong completeness: F{\'e}r{\'e}e and Shillito \cite{feree-shillito} attribute the example showing this to Mojtaba~Mojtahedi. In fact, the author originally found the semantic arguments above by constructing a translation of \textbf{LJ2} into the modal logic \textbf{KM} which sends $\Delta$ to $\square$. In this final section, we provide a brief outline of this argument via proof-theoretic translation. Many of the proofs are left as sketches or omitted altogether.
\end{body}

\begin{definition}[\cite{kuznetsov-km-logic}]
The modal logic \textbf{KM} is an intuitionistic (quantifier-free) propositional modal logic with modality $\square$ and with the modal axioms
\begin{enumerate}
\item $\varphi_1 \rightarrow \square \varphi_1$,
\item $(\square\varphi_1 \rightarrow \varphi_1) \rightarrow \varphi_1$, and
\item $\square \varphi_1 \rightarrow (\varphi_2 \vee (\varphi_2 \rightarrow\varphi_1))$
\end{enumerate}
for all modal propositional formulas $\varphi_1, \varphi_2$.
\end{definition}

\begin{body}
F{\'e}r{\'e}e and Shillito~\cite{feree-shillito} present a cut-free sequent calculus \textbf{G4KM} for \textbf{KM}, which they use to prove a uniform interpolation theorem. We recall the statement as Theorem~\ref{thm:feree-shillito-quantifiers} below, but refer the interested reader to their article for all other details. In what follows, we use $\vdashkm$ to denote the turnstile of the calculus \textbf{G4KM}.
\end{body}

\begin{theorem}[F{\'e}r{\'e}e-Shillito \cite{feree-shillito}]\label{thm:feree-shillito-quantifiers}
Consider a finite sequence of propositional variables $\overline{X}$, and a propositional variable $Y$ outside the sequence $\overline{X}$. Let $\Phi(\overline{X}, Y)$ denote a \textbf{KM}-formula containing only the variables in $\overline{X}, Y$ as free variables. Then one can find quantifier-free \textbf{KM}-formulas
$$\pite Y.\Phi(\overline{X}, Y)$$
and
$$\pita Y.\Phi(\overline{X}, Y)$$
so that the following hold:
\begin{enumerate}
    \item All free propositional variables in $\pite Y.\Phi(\overline{X}, Y)$ and $\pita Y.\Phi(\overline{X}, Y)$ belong to the sequence $\overline{X}$.
    \item For a formula $\Psi(\overline{X})$, \textbf{KM} proves that
    $(\pite Y. \Phi(\overline{X},Y)) \vdashkm \Psi(\overline{X})$ precisely if it proves that $\Phi(\overline{X},Y) \vdashkm \Psi(\overline{X})$.
    \item For a formula $\Psi(\overline{X})$, \textbf{KM} proves that
    $\Psi(\overline{X}) \vdashkm (\pita Y. \Phi(\overline{X},Y))$ precisely if it proves that $\Psi(\overline{X}) \vdashkm \Phi(\overline{X},Y)$, and
    \item $\pita Y. \Phi(\overline{X},Y) \vdashkm \Phi(\overline{X},Y)$ and $\Phi(\overline{X},Y) \vdashkm \pite Y. \Phi(\overline{X},Y)$ are both provable in \textbf{KM} as well.
\end{enumerate}
\end{theorem}

\begin{corollary}\label{cor:feree-shillito-quantifiers}
Let $\Phi(\overline{X}, Y)$ as in Theorem~\ref{thm:feree-shillito-quantifiers} and let $T$ be an arbitrary formula of \textbf{KM}, with no restriction on its variables. As usual, let $\Phi(\overline{X},T)$ denote the result of substituting $T$ for $Y$ throughout $\Phi$ in a capture-avoiding way. Then both
$$\pita Y. \Phi(\overline{X},Y) \vdashkm \Phi(\overline{X},T)\text{ and }\Phi(\overline{X},T) \vdashkm \pite Y. \Phi(\overline{X},Y)$$
are derivable.
\end{corollary}

\begin{definition}\label{def:kappa-formula-translation}
Define a translation $(-)^\kappa$ between \textbf{LJ2}-formulas and \textbf{KM}-formulas by recursion on formula structure using the following cases:
\begin{enumerate}
    \item $X^\kappa := X$ for any propositional variable $X$,
    \item $\bot^\kappa := \bot$,
    \item $(\varphi_1 \wedge \varphi_2)^\kappa := \varphi_1^\kappa \wedge \varphi_2^\kappa$,
    \item $(\varphi_1 \vee \varphi_2)^\kappa := \varphi_1^\kappa \vee \varphi_2^\kappa$,
    \item $(\varphi_1 \rightarrow \varphi_2)^\kappa := \varphi_1^\kappa \rightarrow \varphi_2^\kappa$,
    \item $(\forall Y. \varphi_1)^\kappa :=\  \pita Y. \varphi_1^\kappa$, and
    \item $(\exists Y. \varphi_1)^\kappa :=\  \pite Y. \varphi_1^\kappa$.
\end{enumerate}
We call this the \textit{$\kappa$-translation}.
\end{definition}

\begin{body}
An important feature of \textbf{KM}, not shared with modal logic \textbf{K}, is the derivability of modal congruuence,
$$ (\varphi_1 \leftrightarrow \varphi_2), \square\varphi_1 \vdashkm \square\varphi_2.$$
for all modal formulas $\varphi_1, \varphi_2$. Since the other connectives are extensional, this allows one to substitute equivalent formulas for equivalent formulas.
\end{body}

\begin{lemma}\label{lem:kappa-substitution}
The $\kappa$-translation introduces no new free variables. Moreover, after renaming bound variables, the translation commutes with substitution up to provable equivalence. In other words, we have
$$(\varphi[T/Y])^\kappa \vdashkm \varphi^\kappa[T^\kappa/Y]$$
and
$$\varphi^\kappa[T^\kappa/Y] \vdashkm (\varphi[T/Y])^\kappa$$
as well.
\begin{proof}
Follows the proof of Lemma~8~in~\cite{pitts-quantifiers}, using modal congruence.
\end{proof}
\end{lemma}

\begin{proposition}\label{prop:kappa-proof-translation}
The formula translation of Definition~\ref{def:kappa-formula-translation} extends into a translation of derivations from \textbf{LJ2} to \textbf{G4KM}. In particular, let $\Gamma$ denote a set of \textbf{LJ2}-formulas, and $\varphi$ a single \textbf{LJ2}-formula. If $\Gamma \vdash_2 \varphi$, then $\Gamma^\kappa \vdashkm \varphi^\kappa$.
\begin{proof}
If $\Gamma \vdash_2 \varphi$, we can find some finite sequence $\Lambda$ so that $\forall \lambda \in \Lambda. \lambda \in \Gamma$ and $\Lambda \vdash_2 \varphi$ has an \textbf{LJ2} proof tree. It's enough to establish the theorem for derivations of this form. We induct on the last rule of \textbf{LJ2} proof trees. After proving the rules of \textbf{LJ} admissible for \textbf{G4KM}, the structural and connective cases are routinie. We deal with the quantifier cases below.

\vspace{0.5em} \textbf{Universal left rule.} In this case we have an \textbf{LJ2} proof tree with last rule
$$\dfrac{\Lambda, \varphi[T/Y] \vdash_2 \psi}{\Lambda, \forall Y.\varphi \vdash_2 \psi}\forall L$$
and by induction hypothesis we have a \textbf{G4KM} proof tree with conclusion $$\Lambda^\kappa, (\varphi[T/Y])^\kappa \vdashkm \psi^\kappa.$$
By Lemma~\ref{lem:kappa-substitution} above, $\varphi^\kappa[T^\kappa/Y] \vdashkm (\varphi[T/Y])^\kappa$, so by cut we obtain $\Lambda^\kappa, \varphi^\kappa[T^\kappa/Y] \vdashkm \psi^\kappa$. Use the admissibility of weakening to obtain a proof of $$\Lambda^\kappa, \pita Y.\varphi, \varphi^\kappa[T^\kappa/Y] \vdashkm \psi^\kappa.$$ Cut this against $\pita Y.\varphi^\kappa \vdashkm \varphi^\kappa[T^\kappa/Y]$ (obtained from Corollary~\ref{cor:feree-shillito-quantifiers}) to construct a proof tree for $\Lambda^\kappa, \pita Y.\varphi^\kappa \vdashkm \psi^\kappa$, and hence $\Lambda^\kappa, (\forall Y.\varphi)^\kappa \vdashkm \psi^\kappa$.

\vspace{0.5em} \textbf{Existential right rule.} Similar to (but simpler than) the universal left rule.

\vspace{0.5em} \textbf{Universal right rule.} In this case $\varphi$ has the form $\forall X. \psi$. Our \textbf{LJ2} proof tree has last rule
$$\dfrac{\Lambda \vdash_2 \psi}{\Lambda \vdash_2 \forall Y. \psi}$$
and so we know that $Y$ does not occur free in $\Lambda$. The induction hypothesis gives $\Lambda^\kappa \vdashkm \psi^\kappa$, and since the $\kappa$-translation preserves free variables, $Y$ does not occur in $\Lambda^\kappa$ at all. Thus, by clause 3 of Theorem~\ref{thm:feree-shillito-quantifiers}, $\Lambda^\kappa \vdashkm \psi^\kappa$ holds precisely if $\Lambda^\kappa \vdashkm \pita Y. \psi^\kappa$ does. By definition of $\kappa$, this means $\Lambda^\kappa \vdashkm (\forall Y. \psi)^\kappa$ as required.

\vspace{0.5em} \textbf{Existential left rule.} The \textbf{LJ2} proof tree ends with the rule
$$\dfrac{\Lambda, \varphi \vdash_2 \psi}{\Gamma, \exists Y. \varphi \vdash_2 \psi}\exists L$$
and the induction hypothesis gives us $\Lambda^\kappa, \varphi^\kappa \vdashkm \psi^\kappa$ and thus that $Y$ does not appear free in $\Gamma, \Psi$. Use this to obtain a proof of
$$\varphi^\kappa \vdashkm \bigwedge_{\lambda \in \Lambda^\kappa} \lambda \rightarrow \psi^\kappa.$$
Since the right-hand side is $Y$-free, it follows from clause 2 of Theorem~\ref{thm:feree-shillito-quantifiers} that $$\pite Y.\varphi^\kappa \vdashkm \bigwedge_{\lambda \in \Lambda^\kappa} \lambda \rightarrow \psi^\kappa$$ as well, and by cutting against $$\Lambda^\kappa, \bigwedge_{\lambda \in \Lambda^\kappa} \lambda \rightarrow \psi^\kappa \vdashkm \psi^\kappa,$$
we conclude $\Lambda^\kappa, \pite Y.\varphi^\kappa \vdashkm \psi^\kappa$ and therefore $\Lambda, (\exists Y.\varphi)^\kappa \vdashkm \psi^\kappa$ as required.
\end{proof}
\end{proposition}

\begin{lemma}\label{lem:kappa-translations-of-t-formulas}
The $\kappa$-translation of second-order formula $\Delta \varphi$ is equivalent to $\square \varphi^\kappa$, whereas the $\kappa$-translation of the second-order formula $M$ is equivalent to $\top$.
\end{lemma}

\begin{proof}[Alternative proof of Lemma~\ref{lem:free-nabla-not-derivable}]
If $\mathcal{T} \vdash_2 P_0$, then a finite prefix of $\mathcal{T}$ suffices to derive $P_0$. Passing through the $\kappa$-translation and applying Lemma¬\ref{lem:kappa-translations-of-t-formulas}, we conclude that there exists some $N \in \mathbb{N}$ such that
$$\top, \square P_1\rightarrow P_0,\dots, \square P_N\rightarrow P_{N-1} \vdashkm P_0.$$

Substituting $\square^{N-i}\bot$ for each $P_i$ turns the antecedent into identity implications and the conclusion into $\square^N\bot$. Thus, if one had  $\mathcal{T} \vdash_2 P_0$,  one would have a derivation of $\vdashkm \square^N \bot$. But one can see directly, by analyzing the possible cut-free proofs in the calculus \textbf{G4KM}, that $\vdashkm \square^N \bot$ is not provable for any $N \in \mathbb{N}$.
\end{proof}

\begin{body}
Since proving Proposition~\ref{prop:nabla-characterization-upsets} requires the law of excluded middle, the original proof of Lemma~\ref{lem:kappa-translations-of-t-formulas} does not go through in a constructive metatheory as written. As usual, the proof-theoretic approach is more resilient to changes in metatheory, and the alternative proof above does yield a valid proof of Theorem~\ref{thm:cha-not-complete} in any sufficiently strong constructive metatheory (such as IZF).
\end{body}

\paragraph{Acknowledgments.} It is a particular pleasure to have used the point-free coderivative in this article: I met Harold Simmons several times when I was a student in Manchester, and remember him fondly. I would also like to thank Vincent Jackson for helpful discussions and suggestions.

\bibliography{biblio}
\bibliographystyle{plainurl}

\end{document}